\documentclass[12pt]{article}

\usepackage{mathpazo}
\usepackage{times}

\usepackage{amsthm}
\usepackage{amsmath}  
\usepackage{amsfonts}
\usepackage{amssymb}  
\usepackage{mathrsfs} 

\usepackage{cite}     

\usepackage{tikz}

\usepackage{fullpage}

\title{\Huge\bf 
Nontrivial $t$-Designs over Finite Fields\\[0.25ex] Exist for All $t$\\[2.00ex]}

\author{
Arman Fazeli\\
   \small University of California San Diego\vspace*{-1.0ex}\\
   \small 9500 Gilman Drive, La Jolla, CA\,92093\vspace*{-1.0ex}\\
   \small\it afazelic@ucsd.edu\\[4.5ex]
\and
{Shachar Lovett}\\
   \small University of California San Diego\vspace*{-1.0ex}\\
   \small 9500 Gilman Drive, La Jolla, CA\,92093\vspace*{-1.0ex}\\
   \small\it slovett@ucsd.edu\\[4.5ex]
\and
{Alexander Vardy}\\
   \small University of California San Diego\vspace*{-1.0ex}\\
   \small 9500 Gilman Drive, La Jolla, CA\,92093\vspace*{-1.0ex}\\
   \small\it avardy@ucsd.edu\\[6.5ex]
}

\date{\today\vspace{6.00ex}}

\newtheorem{theorem}{Theorem}

\newtheorem{lemma}[theorem]{Lemma}

\newcommand{\Tref}[1]{The\-o\-rem\,\ref{#1}}

\newcommand{\Lref}[1]{Lem\-ma\,\ref{#1}}
\newcommand{\Cref}[1]{Co\-ro\-lla\-ry\,\ref{#1}}

\newcommand{\be}[1]{\begin{equation}\label{#1}}
\newcommand{\ee}{\end{equation}} 
\newcommand{\eq}[1]{(\ref{#1})}

\renewcommand{\Bbb}{\mathbb}

\newcommand{\F}{{\Bbb F}}

\newcommand{\Z}{{\Bbb Z}}
\newcommand{\Q}{{\Bbb Q}}

\newcommand{\Fq}{{{\Bbb F}}_{q}}
\newcommand{\Span}[1]{{\left\langle {#1} \right\rangle}}

\makeatletter
\DeclareRobustCommand{\sbinom}{\genfrac[]\z@{}}
\makeatother
\newcommand{\G}[2]{\sbinom{{#1}\kern-.05pt}{{#2}\kern-.05pt}}
\newcommand{\Gq}[2]{\sbinom{{#1}}{{#2}}_q}
\DeclareMathOperator{\GF}{GF}

\newcommand{\deff}{\mbox{$\stackrel{\rm def}{=}$}}

\renewcommand{\le}{\leqslant}
\renewcommand{\leq}{\leqslant}
\renewcommand{\ge}{\geqslant}
\renewcommand{\geq}{\geqslant}

\DeclareMathAlphabet{\mathbfsl}{OT1}{ppl}{b}{it} 

\newcommand{\eee}{\mathbfsl{e}}

\makeatletter

\gdef\@punct{.\ \ }  
\def\@sect#1#2#3#4#5#6[#7]#8{%
  \ifnum #2>\c@secnumdepth
     \def\@svsec{}
  \else
     \refstepcounter{#1}\edef\@svsec{%
     \ifnum #2>0{{\csname the#1\endcsname}}.\fi%
    \hskip .5em}
  \fi
  \@tempskipa #5\relax
  \ifdim \@tempskipa>\z@
     \begingroup #6\relax
       \@hangfrom{\hskip #3\relax\@svsec}{\interlinepenalty \@M #8\par}
     \endgroup
     \csname #1mark\endcsname{#7}
     \addcontentsline{toc}{#1}{\ifnum #2>\c@secnumdepth\else
          \protect\numberline{\csname the#1\endcsname}\fi#7}
  \else
     \def\@svsechd{#6\hskip #3\@svsec #8\@punct\csname #1mark\endcsname{#7}
     \addcontentsline{toc}{#1}{\ifnum #2>\c@secnumdepth \else
          \protect\numberline{\csname the#1\endcsname}\fi#7}}
  \fi
  \@xsect{#5}}

\def\@ssect#1#2#3#4#5{\@tempskipa #3\relax
  \ifdim \@tempskipa>\z@
    \begingroup #4\@hangfrom{\hskip #1}{\interlinepenalty \@M #5\par}\endgroup
  \else \def\@svsechd{#4\hskip #1\relax #5\@punct}\fi
  \@xsect{#3}}

\makeatother

\begin{document}

\maketitle

\thispagestyle{empty}

\begin{abstract}
\noindent
A {$t$-$(n,k,\lambda)$ design over $\F_q$}
is a collection of $k$-dimensional subspaces of $\F_q^n$,
called blocks, such that each
$t$-dimensional subspace of $\F_q^n$ is contained in exactly
$\lambda$ blocks. Such $t$-designs over $\F_q$ are the $q$-analogs
of conventional combinatorial designs.
Nontrivial $t$-$(n,k,\lambda)$ designs over $\F_q$ are
currently known to exist only for $t \le 3$. Herein, we prove
that simple (meaning, without repeated blocks) nontrivial 
$t$-$(n,k,\lambda)$ designs over $\F_q$ exist for all $t$ and $q$,
provided that $k > 12t$ and $n$ is sufficiently large.
This may be regarded as a $q$-analog of the celebrated 
Teirlinck theorem for combinatorial designs.
\end{abstract}

\newpage
\vspace{3.00ex}
\section{Introduction}

\renewcommand{\baselinestretch}{1.18}\normalsize
Let $X$ be a set with $n$ elements. A \emph{$t$-$(n,k,\lambda)$
combinatorial design} (or {\emph{$t$-design}}, in brief) is 
a~collection of $k$-subsets of $X$, called blocks, such that each
$t$-subset of $X$ is contained in exactly $\lambda$ blocks. A $t$-design
is said to be \emph{simple} if there are no repeated blocks --- that is,
all the $k$-subsets in the collection are distinct.
A \emph{trivial $t$-design} is the set of all $k$-subsets of $X$.  
The celebrated~theorem of Teirlinck~\cite{Tei87} establishes the 
existence of nontrivial simple $t$-designs for all $t$.

It was suggested by Tits~\cite{Tits57} in 1957 that combinatorics 
of sets could be regarded as the limiting case $q\,{\to}\,1$ of 
combinatorics of vector spaces over the finite field $\F_q$. Indeed, 
there is a strong analogy between subsets
of a set and subspaces of a vector space, expounded by numerous
authors~\cite{Cohn,GR,Wang}. In particular, the
notion of $t$-designs has been extended to vector spaces
by Cameron~\cite{Cameron1,Cameron2} and Delsarte~\cite{Delsarte} 
in the early 1970s. Specifically, let $\F_q^n$ be a vector space 
of dimension $n$ over the finite field $\F_q$.
Then a \emph{$t$-$(n,k,\lambda)$ design over\/ $\F_q$}
is a collection of $k$-dimensional~subspaces of $\F_q^n$
($k$-subspaces, for short), called blocks, such that each
$t$-subspace of $\F_q^n$ is contained in exactly $\lambda$ blocks.
Such $t$-designs over $\F_q$ are the $q$-analogs of conventional
combinatorial designs. As for combinatorial designs, we will say
that a $t$-design over $\Fq$ is \emph{simple} if it does not
have repeated blocks, and \emph{trivial} if it is the set 
of all $k$-subspaces of $\F_q^n$.

The first examples of simple
nontrivial $t$-designs over $\F_q$ with $t \geq 2$ were
found by Thomas~\cite{T87} in 1987. Today, following the work of many
authors~\cite{BKL,MMY,Chaudhuri,S90,S92,T96,KOW13}, numerous such
examples are known. All these examples have $t=2$ or $t=3$.
If repeated blocks are allowed, nontrivial $t$-designs over
$\Fq$ exist for all $t$, as shown in~\cite{Chaudhuri}.
 However, no simple nontrivial $t$-designs
over $\F_q$ are presently known for $t>3$. 
Our main result is the following theorem.

\begin{theorem}
\label{thm:main}
Simple nontrivial $t$-$(n,k,\lambda)$ designs
over $\F_q$ exist for all $q$ and $t$, and all $k > 12(t{+}1)$
provided that $n \ge c k t$ 
for a large enough absolute constant $c$. 
Moreover, these $t$-$(n,k,\lambda)$ designs
have at most $q^{12(t+1)n}$ blocks.
\end{theorem}

This theorem can be regarded as a $q$-analog of Teirlinck's 
theorem~\cite{Tei87} for combinatorial designs.
Our proof of \Tref{thm:main}
is based on a new probabilistic technique introduced by Kuperberg,
Lovett, and Peled in~\cite{KLP12} to prove the existence of certain
regular combinatorial structures. We~note that this proof technique 
is purely existential: there is no known efficient algorithm which 
can produce $t$-$(n,k,\lambda)$ design over $\Fq$ for $t > 3$.
Hence, we pose the following as an open problem:
\renewcommand{\theequation}{$\star$}
\be{openproblem}
\text{\it Design an efficient algorithm to
produce simple nontrivial $t$-$(n,k,\lambda)$ designs for large $t$}
\ee

\addtocounter{equation}{-1}
\renewcommand{\theequation}{\arabic{equation}}

The rest of this paper is organized as follows.
We begin with some preliminary definitions~in~the next section. 
We present the Kuperberg-Lovett-Peled (KLP) theorem of \cite{KLP12}
in Section\,\ref{sec:KLP}. In Section\,\ref{sec:main}, 
we apply this theorem to prove the existence of simple $t$-designs 
over $\Fq$ for all $q$ and $t$. 
Detailed proofs of some of the technical lemmas are 
deferred to Section\,\ref{sec:technical_lemmas}.

\section{Preliminaries}
\label{sec:prelim}

Let $\F_q$ denote the finite field with $q$ elements, and let $\F_q^n$ 
be a vector space of dimension $n$ over~$\F_q$. We recall some basic 
facts that relate to counting subspaces of $\F_q^n$.
The number of distinct $k$-sub\-spaces of~$\F_q^n$ is given by
the \emph{$q$-binomial (a.k.a.\ Gaussian) coefficient}
\be{Gq}
\Gq{n}{k}
\ \deff\ \
\frac{[n]_q!}{[k]_q!\,[n-k]_q!}
\ee
where $[n]_q!$ is the \emph{$q$-factorial} defined by
\be{nq!}
[n]_{q}!
\:\ \deff\ \
[1]_{q}[2]_{q} \ldots [n]_{q}
\ = \
\bigl(1+q\bigr)\bigl(1+q+q^2\bigr)\cdots\bigl(1+q+q^{2}+ \cdots +q^{n}\bigr)
\ee
Observe the similarities between \eq{Gq} and \eq{nq!}
and the conventional binomial coefficients and factorials,
respectively. Many more similarities between the combinatorics
of sets and combinatorics of vector spaces are known; see~\cite{KC},
for example. Here, all we need are upper and lower bounds on
$q$-binomial coefficients, established in the following lemma.
\begin{lemma}
\label{lemma:bounds}
$$
q^{k(n-k)}
\: \leq \
\Gq{n}{k} 
\leq \:
{n\choose k}q^{k(n-k)}
\vspace{0.75ex}
$$
\end{lemma}

\begin{proof}
We use the following identity from~\cite[p.\,19]{KC},
\be{KC}
\Gq{n}{k}=\sum_{1\leq s_1 <s_2< \cdots <s_k\leq n} q^{(s_1+s_2+ \ldots +s_k)-{k(k+1)/2}}
\ee
The largest term in the sum of \eq{KC} is $q^{k(n-k)}$, which
corresponds to $s_i=n-k+i$ for all $i$. The number of terms 
in the sum is ${n \choose k}$, and the lemma follows.
\end{proof}

\vspace{3ex}
\section{The KLP theorem}
\label{sec:KLP}

Kuperberg, Lovett, and Peled~\cite{KLP12} developed
a powerful probabilistic method to prove the existence 
of certain regular combinatorial structures, such as {orthogonal arrays}, 
{combinatorial designs}, and $t$-wise permutations. In this section, 
we describe their main theorem.

Let $M$ be a $|B| \times |A|$ matrix with integer entries, 
where $A$ and $B$ are the set of columns and the set of rows 
of $M$, respectively. We think of the elements of $A$, 
respectively $B$, as vectors in $\Z^B$, respectively
in $\Z^A$. We are interested in those matrices $M$
that satisfy the five properties below.
\begin{itemize}
\item[\bf 1.]
{\bf Constant vector.} 
There exists a rational linear combination of the columns of $M$
that produces the vector $(1,1, \ldots ,1)^T$.

\item[\bf 2.]
{\bf Divisibility.} Let $\overline{b}$ 
denote the average of the rows of $M$, namely 
$
\overline{b}= \frac{1}{|B|} \sum_{b\in B} b
$.
There is an integer $c_1 \,{<}\, |B|$ such that 
the vector $c_{1} \overline{b}$ can be produced as
an integer linear combination of the rows of $M$. 
The smallest such $c_1$ is called the \textit{divisibility parameter}.

\item[\bf 3.]
{\bf Boundedness.} 
The absolute value of all the entries in $M$ is bounded
by an integer $c_2$, which is called the \textit{boundedness parameter}.

\item[\bf 4.]
{\bf Local decodability.} 
There exist a positive integer $m$ and an integer $c_3 \geq m$ 
such that, for every column $a\in A$, there is a vector
of coefficients 
$\gamma^{a}=(\gamma_1,\gamma_2\ldots,\gamma_{|B|}) \in \Z^B$ satisfying
$||\gamma^{a}||_1\leq c_3$ and 
$
\sum_{b\in B} \gamma_{b} b
 =  
m \eee_a
$, 
where $\eee_a \in \{0,1\}^A$ is the vector with $1$ 
in coordinate $a$~and $0$ in all other coordinates.
The parameter $c_3$ is called the \textit{local decodability parameter}.

\item[\bf 5.]
{\bf Symmetry.} \looseness=-1
A \textit{symmetry} of the matrix $M$ is a permutation of rows $\pi \in S_B$
for which there exists an invertible linear map $\ell:\Q^A\to
\Q^A$ such that applying the permutation on rows and the linear map on
columns does not change the matrix, namely $\ell(\pi(M))=M$. The group
of symmetries of $M$ is denoted by $Sym(M)$. It is required that this
group acts transitively on~$B$. That is, for all $b_1,b_2\in B$ there
exists a permutation $\pi \in Sym(M)$ satisfying $\pi(b_1)=b_2$.
\end{itemize}

The following theorem has been proved by Kuperberg, Lovett, and Peled
in \cite{KLP12}. In fact, the results of Theorem\,2.4 and Claim\,3.2 
of~\cite{KLP12} are 
more general than \Tref{thm:klp} below.
However, \Tref{thm:klp} will suffice for our purposes.

\begin{theorem}
\label{thm:klp}
Let $M$ be a $|B| \times |A|$ integer matrix satisfying the
five properties above. Let $N$ be an integer divisible by $c_1$ such that
\be{KLP-bound}
c |A|^{52/5}c_{1}(c_{2}c_3)^{12/5}\log \bigl(|A|c_2\bigr)^8
\,\leq N \,<|B|
\ee
where $c>0$ is a sufficiently large absolute constant. \looseness=-1
Then there exists a set of rows $T\subset B$ of size
$|T|=N$ such that the average of the rows in $T$ is
equal to the average of all the rows in $M$, namely 
\be{KLP}
\frac{1}{N}\sum\limits_{b\in T} b
\ = \
\frac{1}{|B|}\sum\limits_{b\in B} b
\ = \
\overline{b}
\ee
\end{theorem}

\vspace{3ex}
\section{Proof of the main result}
\label{sec:main}

We will apply Theorem~\ref{thm:klp} to prove existence of designs over
finite fields. We first introduce~the~appropriate matrix $M$, which is
the incidence matrix of $t$-subspaces and $k$-subspaces.\pagebreak[3.99]

Let $M$ be a $|B|\times |A|$ matrix,
whose columns $A$ and rows $B$ correspond to
the $t$-subspaces and the $k$-subspaces of $\F_q^n$, 
respectively. Thus
$
|A| = \Gq{n}{t}
$
and
$ 
|B| = \Gq{n}{k}
$.
The entries of $M$ are defined by
$M_{b,a} = 1_{a \subset b}$. 
It is easy to see that a simple $t$-$(n,k,\lambda)$ design over $\F_q$ 
corresponds to a set of rows $b_{1},b_{2},\ldots,b_{N}$ of $M$ 
such that 
\be{design-def}
b_{1} + b_{2} + \cdots + b_{N}
\ = \
(\lambda,\lambda,\ldots,\lambda)
\hspace{5ex}\text{for some $\lambda \in \mathbb{N}$}
\ee
Note that this implies $\lambda \Gq{n}{t} = N \Gq{k}{t}$,
because each row $b \in B$ has Hamming weight $\Gq{k}{t}$.
In~order to relate \eq{design-def} to 
\Tref{thm:klp}, we need the following simple lemma.
The lemma is well known; we include a brief
proof for completeness.

\begin{lemma}
\label{lemma:I}
Let $V$ be a $t$-subspace of\/ $\F_q^n$. 
The number of $k$-subspaces $U$ 
such that\/ $V\subset U\subset \F_q^n$ is~given 
by\/ $\Gq{n-t}{k-t}$.
\end{lemma}

\begin{proof}
Fix a basis $\{v_1,v_2, \ldots ,v_t\}$ for $V$.
We extend this basis to a basis $\{v_1,v_2, \ldots ,v_k\}$
for $U$. The number of ways to do so is
$(q^n-q^{t})(q^n-q^{t+1}) \cdots (q^n-q^{k-1})$. However, each
subspace $U$ that contains $V$ is counted 
$(q^k-q^{t})(q^k-q^{t+1}) \cdots (q^k-q^{k-1})$ 
times in the above expression.
\end{proof}

It follows from \Lref{lemma:I} that
\be{eq:avgb}
\overline{b}
\ = \
\frac{1}{|B|}\sum\limits_{b\in B} b
\ = \
\frac{\displaystyle\Gq{n-t}{k-t}}{\displaystyle\Gq{n}{k}} 
\,\bigl(1,1,\ldots,1\bigr)
\ = \
\frac{\displaystyle\Gq{k}{t}}{\displaystyle\Gq{n}{t}} 
\,\bigl(1,1,\ldots,1\bigr)
\ee
Therefore, a simple nontrivial $t$-$(n,k,\lambda)$ design over $\F_q$ 
is a set of $N < |B|$ rows of $M$ satisfying
$$
b_{1}+b_{2}+ \cdots +b_{N} \ = \ N \overline{b}
$$
But this is precisely the guarantee provided by
\Tref{thm:klp} in \eq{KLP}. 
Note that the corresponding value 
of $\lambda = N \Gq{k}{t}/\Gq{n}{t}$
would be generally quite large.

\vspace{2ex}
\subsection{Parameters for the KLP theorem}

Let us now verify that the matrix $M$ satisfies the five conditions 
in \Tref{thm:klp} and estimate the relevant parameters $c_1,c_2,c_3$ in 
\eq{KLP-bound}.

\paragraph{Constant vector} \hspace*{-3ex}
Each $k$-subspace contains exactly $\Gq{k}{t}$ 
$t$-subspaces, so the sum of all the columns of $M$ 
is $\Gq{k}{t}(1,\ldots,1)^T$. 
Hence $(1,1,\ldots,1)^T$ 
is a rational linear combination of the columns of $M$.

\paragraph{Symmetry} \hspace*{-3.5ex} 
An invertible linear transformation $L:\F_q^n\to \F_q^n$ acts
on the set of $k$-subspaces~by mapping $U = \Span{v_1,v_2,\ldots,v_k}$ to
$L(U)= \Span{L(v_1),L(v_2) \ldots, L (v_k)}$.
It acts on the set of $t$-sub\-spaces in the same way.
Note that if $U$ is a $k$-sub\-space and $V$ is a
$t$-subspace, then $V\subset U$ if and only if $L(V)\subset L(U)$. 
Now, let $\pi_L\in S_B$ be the permutation of rows of $M$ induced 
by~$L$,~and~let
$\sigma_L\in S_A$ be the permutation of columns of $M$ induced by
$L$. Then $\pi_L\bigl(\sigma_L(M)\bigr)=M$. Note that~$\sigma_L$ acts as an
invertible linear map on $\Q^A$ by permuting the coordinates.  Hence,
$\pi_L$ is a symmetry of~$M$. 
The corresponding symmetry
group is, in fact, the general linear group ${\rm GL}(n,q)$.
It is well known that ${\rm GL}(n,q)$ is transitive:
for any two $k$-subspaces $U_1,U_2$, we 
can find an invertible linear transformation $L$ such that
$L(U_1)=U_2$, which implies $\pi_L(b_1) = b_2$ for the corresponding rows.
\vspace{-1.0ex}

\paragraph{Boundedness} \hspace*{-3ex}
Since all entries of $M$ are either $0$ or $1$, we can set $c_2=1$.
\vspace{-1.0ex}

\paragraph{Local decodability} \hspace*{-3ex}
Let $m$ be a positive integer to be determined later. Fix a
$t$-subspace $V$ corresponding to a column of $M$. We wish to find a
short integer combination of rows of $M$ summing to $m \eee_V$. In
order to do so, we fix an arbitrary $(t+k)$-subspace $W$ that contains
$V$. As part of the short integer combination, we will only choose 
those rows that correspond to the $k$-subspaces contained in $W$. 
Moreover, the integer coefficient for a $k$-subspace $U \subset W$ will 
depend only on the dimension $j=\dim(U \cap V)$.
We denote this coefficient by $f_{k,t}(j)$. 

We need the following conditions to hold. 
First, by \Lref{lemma:I}, there are
$\Gq{k}{k-t}$ $k$-subspaces $U$ such that $V \subset U \subset W$.
Therefore, we need
\begin{align}
\label{eq:cond1}
f_{k,t}(t) \Gq{k}{k-t}
= \
m
\end{align}
Second, for any other $t$-subspace $V' \subset \F_q^n$, we need that
\begin{align}\label{eq:cond2a}
\sum_{V' \subset U \subset W} f_{k,t}\bigl(\dim(U \cap V)\bigr)
\ = \
0
\end{align}
where the sum is over all $k$-subspaces $U$ containing $V'$ and contained
in $W$. Note that we only need to consider those $t$-subspaces $V'$ that 
are contained in $W$. For all other $t$-subspaces, our integer combination 
of rows of $M$ produces zero by construction.

The following lemma counts the number of $k$-subspaces which contain
$V'$ and whose intersection with $V$ has a prescribed dimension. Its
proof is deferred to Section~\ref{sec:technical_lemmas}.

\begin{lemma}
\label{lemma:II}
Let $V_1,V_2$ be two distinct $t$-subspaces of\/ $\F_{q}^{n}$ 
such that\/ $\dim(V_1 \cap V_2)=l$ for some $l$ in $\{0,1, \ldots ,t-1\}$. 
The number of $k$-subspaces $U\subset \F_{q}^{n}$ such that
$V_1 \subset U$ and $\dim(U \cap V_2) = j$, for some
$j \in \{l,l+1,\ldots,t\}$,
is given by 
\be{lemma2}
q^{(k-t-j+l)(t-j)} \Gq{t-l}{j-l}\Gq{n-2t+l}{k-t-j+l} 
\ee
\end{lemma}

With the help of \Lref{lemma:II} we
can rephrase~\eq{eq:cond2a} as the following set of $t$
linear equations:
\begin{align}\label{eq:cond2}
\sum_{j=l}^{t}\,f_{k,t}(j)\Gq{t-l}{t-j}\Gq{k-t+l}{j}q^{(k-t-j+l)(t-j)}
\ = \
0
\hspace{6ex}
\text{for~ $l = 0,1,\ldots,t-1$}
\end{align}
Equations \eqref{eq:cond1} and \eqref{eq:cond2} together form 
a set of $t+1$ linear equations, which can be represented 
in the form of a matrix production:
\begin{align}\label{eq:mform}
D f
\ = \
(0,0, \ldots ,0,m)^{T}
\end{align} 
where
$f=\bigl(f_{k,t}(0),f_{k,t}(1), \ldots ,f_{k,t}(t)\bigr)^T$
and
$D$ is an upper-triangular $(t{+}1) \times (t{+}1)$ matrix~with 
entries
\begin{align}
d_{l,j}
\ = \
\Gq{t-l}{t-j}\Gq{k-t+l}{j}q^{(k-t-j+l)(t-j)} 
\hspace{6ex}
\text{for~ $0 \le l \le j \le t$}
\end{align}
The condition $t\leq k$ ensures nonzero values on the main diagonal. 
Therefore, $\det D$ is nonzero and the system of linear equations 
is solvable. By Cramer's rule, we have
\begin{align}
f_{k,t}(j)
\ = \
\frac{\det D_j}{\det D}\: m
\end{align}
where $D_j$ is the matrix formed by replacing 
the $j$-th column of $D$ by the vector $(0,0,\ldots,1)^T$.
Note that $\det D$ is an integer. Thus we set
$m= \det D$, so that $f_{k,t}(j)= \det D_j$.
This guarantees that the coefficients
$f_{k,t}(0),f_{k,t}(1), \ldots ,f_{k,t}(t)$
are integers.

We are now in a position to establish a bound
on the local decodability parameter $c_3$. First,
the following lemma bounds the determinants of $D$ and $D_j$.
We defer its proof to Section~\ref{sec:technical_lemmas}.
\begin{lemma}
\label{lemma:det}$\,$\vspace{-3ex}
\begin{align*}
|\det D|~ &\leq ~q^{k(t+1)^2}
\\
|\det D_j|~ &\leq ~q^{k(t+1)^2}
\hspace{5ex}\text{for~ $j = 0,1,\ldots,t$}
\end{align*}
\end{lemma}

The number of $k$-subspaces $U$ contained in $W$ is $\Gq{k+t}{k}$.
We have multiplied the row of $M$ corresponding to each 
such subspace by a coefficient $f_{k,t}(j)$ which 
is bounded by $q^{k(t+1)^2}$. Hence
\begin{align}
c_3 
\,= \,
\max\bigl\{m,\|f\|_1\bigr\} 
\,\le\,
\Gq{k+t}{k} q^{k(t+1)^2} 
\le\, 
{k+t \choose k} q^{kt} q^{k(t+1)^2} 
\le\,
q^{2k(t+1)^2}
\end{align}

\paragraph{Divisibility} \hspace*{-3.0ex}
The proof of local decodability also makes it possible
to establish a bound~on~the~divisibility parameter $c_1$. 
We already know that for $m=\det D$,
we can represent any element in $m \Z^A$ as an integer 
combination of rows of $M$.
By~\eqref{eq:avgb}, we have 
$
\Gq{n}{t} \overline{b} 
=
\Gq{k}{t}(1,1,\ldots,1)
$. 
Hence, $m \Gq{n}{t} \overline{b} \in m \Z^A$ 
can be expressed as an integer combination
of rows of $M$. It follows that
\begin{align}
c_1 
\,\le\,
m\Gq{n}{t} 
\le \
q^{k(t+1)^2} {n\choose t}q^{t(n-t)} 
\,\le\
q^{k(t+1)^2+t(n-t)+n}
\end{align}

\vspace{2ex}
\subsection{Putting it all together}

We have proved that the incidence matrix $M$ satisfies the five conditions 
in \Tref{thm:klp}, and established the following bounds on the parameters:
\begin{align}
\label{c1}
&c_1 ~\le~ q^{k(t+1)^2+t(n-t)+n}  \\
&c_2 ~=~1\\
&c_3 ~\leq~ q^{2k(t+1)^2}
\end{align}
By \Lref{lemma:bounds}, we also have
\begin{align}
\label{|A|}
|A|=\Gq{n}{t} &\leq~ {n \choose t}q^{t(n-t)}\leq q^{t(n-t)+n}
\\
\label{|B|}
|B| = \Gq{n}{k} &\ge~ q^{k(n-k)}
\end{align}
Combining \eq{KLP-bound} with \eq{c1}\,--\,\eq{|A|},
we see that the lower bound on $N$ in Theorem~\ref{thm:klp} is at most
\be{final-bound}
c'|A|^{52/5}c_{1}(c_{2}c_3)^{12/5}\log (|A|c_2)^8 
\ \leq \ 
cq^{(57/5) \cdot (t+1)n+ckt^2} n^c
\ee
for some absolute constant $c>0$. If we fix $t$ and $k$,
while making $n$ large enough, then the right-hand side of \eq{final-bound}
is bounded by $c q^{12 (t+1) n}$. In view of \eq{|B|}, this is strictly less
than $|B|$ whenever $k>12(t+1)$ and $n$ is large enough.  
It now follows from \Tref{thm:klp} that
for large enough $n$, there exists a simple
$t$-$(n,k,\lambda)$-design over $\F_q$ of size $N \leq cq^{12n(t+1)}$. 
The reader can verify that this holds whenever 
$n \ge \tilde{c} kt$ for a large enough constant $\tilde{c}>0$.

\vspace{3ex}
\section{Proof of the technical lemmas}
\label{sec:technical_lemmas}

In this section, we prove the two technical lemmas (\Lref{lemma:II} 
and \Lref{lemma:det}) we have used to establish the local decodability 
property.

\subsection{Proof of Lemma~\ref{lemma:II}}
Let $V_1,V_2$ be two distinct $t$-subspaces of $\F_{q}^{n}$ 
with $\dim(V_1 \cap V_2)=l$.
Let $U$ be a $k$-subspace~of~$\F_{q}^{n}$ such that
$V_1 \subset U$ and $\dim(U \cap V_2) = j$.
Further, let $X=V_1\cap V_2$ and $Y=V_1+V_2$. 
It~is~not~difficult to show that the following holds:
\begin{align}
\dim(X) &=\,l  & \dim(Y) &=\,2t-l \nonumber \\
\dim(U\cap V_1) &=\,t & \dim(U\cap V_2) &=\,j   \\
\dim(U\cap X) &=\,l &\dim(U\cap Y) &=\, t+j-l &  \nonumber
\end{align}
We will proceed in three steps.
First, fix a basis $\{v_1,v_2, \ldots ,v_t\}$ for $V_1$. 
Next, we extend $V_1$ to the subspace $Z= U\cap Y$ which has 
an intersection of dimension $j$ with $V_2$. In order to do that, 
we pick $j-l$ vectors $v_{t+1},v_{t+2}, \ldots ,v_{t+j-l}$
from $Y\setminus V_1$, 
in such a way that
$v_1,v_2\ldots,v_{t+j-l}$ are linearly independent. 
The number of ways to do so is 
\be{Z-count}
N_1
\ =
\prod_{i=0}^{j-l-1}\Bigl(q^{2t-l}-q^{t+i}\Bigr)
\ee
However, each such subspace $Z$ is counted more than once 
in \eq{Z-count}, 
since there are many different ordered bases for $Z$.
The appropriate normalizing factor is 
$N_2=\prod_{i=0}^{j-l-1}\bigl(q^{t+j-l}-q^{t+i}\bigr)$.
Hence, the total number of different choices for $Z$ is
\begin{align}
\label{N1N2}
\frac{N_1}{N_2}
\ =
\prod_{i=0}^{j-l-1} \frac{q^{2t-l}-q^{t+i}}{q^{t+j-l}-q^{t+i}} 
\ = 
\prod_{i=0}^{j-l-1} \frac{q^{t-l}-q^{i}}{q^{j-l}-q^{i}} 
\ = \
\Gq{t-l}{j-l}
\end{align}
In order to to complete $U$, we need to extend $Z$ by $k-(t+j-l)$
linearly independent vectors 
chosen from $\F_q^n \setminus Y$.
The number of ways to do so is
$N_3=\prod_{i=0}^{k-(t+j-l)-1}\bigl(q^n-q^{(2t-l)+i}\bigr)$, with normalizing
factor \smash{$N_4=\prod_{i=0}^{k-(t+j-l)-1}\bigl(q^k-q^{(t+j-l)+i}\bigr)$}. We have
\begin{align}
\label{N3N4}
\frac{N_3}{N_4}~
&=\!\!\prod_{i=0}^{k-(t+j-l)-1}{q^{(2t-l)+i}\over q^{(t+j-l)+i}} 
\cdot {q^{n-(2t-l)-i}-1\over q^{k-(t+j-l)-i}-1}
\ = \
q^{(k-t-j+l)(t-j)} \Gq{n-2t+l}{k-(t+j-l)}
\end{align}
Combining \eq{N1N2} and \eq{N3N4}, the total number of different
choices for the desired subspace $U$ is given by \eq{lemma2},
as claimed.

\subsection{Proof of Lemma~\ref{lemma:det}}

Lemma~\ref{lemma:det} follows from the following two
lemmas. The first bounds the product of the largest~elements in each
row. The second bounds the number of nonzero generalized diagonals in
$D_j$ --- that is, the number of permutations $\pi \in S_{t+1}$ such that
$(D_j)_{i,\pi(i)} \ne 0$ for all $i \in \{0,1,\ldots,t\}$.

\begin{lemma}\label{lemma:max_element}
$$\prod_{l=0}^t \max_{j} d_{l,j} 
\ \le \
2^{k(t+1)+1} q^{(k-t)t(t+1)}$$
\end{lemma}

\begin{proof}
We first argue that for $l \in \{1,2,\ldots,t\}$,
the largest element in row $l$ is $d_{l,l}$.
For $l=0$, the largest element in the row
is either $d_{0,0}$ or $d_{0,1}$.
To see that, we calculate\pagebreak[3.99]
\begin{align*}
{d_{l,j+1}\over d_{l,j}}
&=~{\displaystyle\Gq{t-l}{t-j-1} \over \displaystyle\Gq{t-l}{t-j}}
\cdot
{\displaystyle\Gq{k-t+l}{j+1}\over \displaystyle\Gq{k-t+l}{j}}
\cdot
q^{(k-t-j+l-1)(t-j-1)-(k-t-j+l)(t-j)}
\\
&=~{[t-j]_{q}![j-l]_{q}!\over [t-j-1]_{q}![j-l+1]_{q}!}
\cdot
{[j]_{q}![k-t+l-j]_{q}!\over [j+1]_{q}![k-t+l-j-1]_{q}!}
\cdot
q^{1-(t-j)-(k-t-j+l)}\\
&=~{q^{t-j}-1 \over q^{j-l+1}-1}
\cdot
{q^{k-t+l-j}-1 \over q^{j+1}-1}
\cdot
q^{1-(t-j)-(k-t-j+l)} \\
&=~{q^{t-j}-1\over q^{t-j}} {q^{k-t-j+l}-1\over q^{k-t-j+l}} {q\over (q^{j+1}-1)(q^{j-l+1}-1)}\\
&<~{q\over (q^{j+1}-1)(q^{j-l+1}-1)}
\end{align*}
Note that unless $j=l=0$, this implies that $d_{l,j+1} < d_{l,j}$. The
only remaining case is $d_{0,1}/d_{0,0} < q/(q-1)^2$. This ratio can
be at most $2$ for $q=2$, and is below $1$ for $q>2$. Hence
$$
\prod_{l=0}^t \max_{j} d_{l,j} 
\ \le \
2 \prod_{j=0}^t d_{j,j}
\vspace{-1.5ex}
$$
We next bound this product:
\begin{align*}
\prod_{j=0}^{t} d_{j,j}&=~\prod_{j=0}^{t}\Gq{k-t+j}{j}q^{(k-t)(t-j)} \le
\prod_{j=0}^{t}{k-t+j\choose j}q^{j(k-t)+(k-t)(t-j)} \le 2^{k(t+1)} q^{(k-t)t(t+1)}
\\[-6.00ex]
\end{align*}
\end{proof}

\vspace{.5ex}

\begin{lemma}\label{lemma:diagonals}
$D_j$ has at most $2^{t}$ nonzero generalized diagonals.
\end{lemma}

\begin{proof}
Let $\pi \in S_n$ be such that $(D_j)_{i,\pi(i)} \ne 0$ for all $i$. If $j>0$ then we must have $\pi(i)=i$ for all $i<j$, and $\pi(t)=j$. Letting $r=t-j$ this reduces to the following problem: let $R$ be an $r \times r$ matrix corresponding to rows $j,\ldots,t-1$ and columns $j+1,\ldots,t$ of $D_j$. This matrix has entries $r_{l,j} \ne 0$ only for $j \ge l-1$. We lemma that such matrices have at most $2^t$ nonzero generalized diagonals. We show this by induction on $r$. Let us index the rows and columns of $R$ by $1,\ldots,r$. To get a nonzero generalized diagonal we must have $\pi(r)=r-1$ or $\pi(r)=r$. In both cases, if we delete the $r$-th row and the $\pi(r)$-th column of $R$, one can verify that we get an $(r-1) \times (r-1)$ matrix of the same form (e.g. zero values in coordinates $(l,j)$ whenever $j<l-1$). The lemma now follows by induction.
\end{proof}

\begin{proof}[Proof of Lemma~\ref{lemma:det}]
The determinant of $D$ or $D_j$ is bounded by the number of nonzero generalized diagonals (which is $1$ for $D$, and at most $2^t$ for $D_j$), multiplied by the maximal value a product of choosing one element per row can take. Hence, it is bounded by
$$
\max\bigl\{|\det D|,|\det D_j|\bigr\}
\le 2^t \cdot 2^{k(t+1)+1} q^{(k-t)t(t+1)} \le q^{t+k(t+1)+1 + (k-t)t(t+1)} \le q^{k(t+1)^2}
\vspace*{-2ex}
$$
\end{proof}\vspace*{-3ex}\pagebreak[3.99]

\section*{Acknowledgment}\vspace{-.75ex}
We are grateful to Michael Braun and Alfred Wasserman
for helpful discussions regarding~the~history 
and the current state of knowledge about
$t$-designs over finite fields.

\end{document}